\UseRawInputEncoding
\documentclass[12pt,reqno]{amsart}
\usepackage{amssymb,amsmath,graphicx,
	amsfonts,euscript}
\usepackage{color}
\usepackage{mathrsfs}
\UseRawInputEncoding
\theoremstyle{plain}

\newtheorem{theorem}{Theorem}[section]
\newtheorem{lemma}[theorem]{Lemma}
\newtheorem{definition}[theorem]{Definition}

\newtheorem{remark}[theorem]{Remark}
\numberwithin{equation}{section}

\begin{document}

\title[Continuity properties and ill-posedness]{Continuity properties of the data-to-solution map and ill-posedness for a two-component Fornberg-Whitham system}

\author[Fei Xu, Yong Zhang, Fengquan Li]{Xu Fei, Zhang Yong, Fengquan Li}

\address[Fei Xu]{School of Mathematical Sciences, Dalian University
of Technology, Dalian, 116024, China}

\address[Yong Zhang]{ School of Mathematical Sciences, Dalian University
of Technology, Dalian, 116024, China} \email{18842629891@163.com}

\address[Fengquan Li]{School of Mathematical Sciences, Dalian University
of Technology, Dalian, 116024, China}

\begin{abstract}
This work studies a two-component Fornberg-Whitham (FW) system, which can be considered as a model for the propagation of shallow water waves. It's known that its solutions depend continuously on their initial data from the local well-posedness result. In this paper, we further show that such dependence is not uniformly continuous in $H^{s}(R)\times H^{s-1}(R)$ for $s>\frac{3}{2}$, but H\"{o}ler continuous in a weaker topology. Besides, we also establish that the FW system is ill-posed in the critical Sobolev space $H^{\frac{3}{2}}(R)\times H^{\frac{1}{2}}(R)$ by proving the norm inflation.
	\end{abstract}
\maketitle

\section{Introduction}
In this paper, we consider the cauchy problem of following two-component Fornberg-Whitham system
 \begin{equation}\label{e11}
 \left\{ \begin{array}{ll}
 u_{t}-u_{txx}+u_{x}+uu_{x}=3u_{x}u_{xx}+uu_{xxx}+\eta_{x}, & t>0, x\in {\rm R}, \\
 \eta_{t}+(\eta u)_{x}=0,  & t>0, x\in {\rm R},\\
 u(x,0)=u_{0}(x), ~~\eta(x,0)=\eta_{0}(x),
 \end{array} \right.
 \end{equation}
where $u=u(x,t)$ describes the horizontal velocity of the fluid and $\eta=\eta(x,t)$ is related to the deviation of the water surface from equilibrium. The system (\ref{e11}) can be written by the nonlocal form
\begin{equation}\label{e12}
 \left\{ \begin{array}{ll}
 u_{t}+uu_{x}=\partial_{x}\Lambda^{-2}(\eta-u), & t>0, x\in {\rm R}, \\
 \eta_{t}+(\eta u)_{x}=0,  & t>0, x\in {\rm R},\\
 u(x,0)=u_{0}(x), ~~\eta(x,0)=\eta_{0}(x),
 \end{array} \right.
 \end{equation}
where $\Lambda:=(I-\partial_{x}^{2})^{\frac{1}{2}}.$ Motivated by generation of two-component Camassa-Holm (CH) system in \cite{16} and two-component Degasperis Procesi (DP) system in \cite{18}, Fan et al. in \cite{17} generalized the Fornberg-Whitham equation to the two-component system. Unlike the CH system (equation) or DP system (equation), they possess infinitely many conserved quantities, a Lax pair and a bi-Hamiltonian structure. The FW system (equation) is not completely integrable (see \cite{24}) and loses some important conversation laws. However, it captures several mathematical
features of the Euler equations, which the KdV equation does not (including nonlocality, wave breaking and highest waves) see \cite{8,25,26}. Hence, it's meaningful to investigate the properties of the FW system (equation) as a good alternative to the KdV equation in water waves.  The travelling wave solutions of FW system (\ref{e12}) were investigated in \cite{17}, where solitary solutions, kink solutions, antikink solutions and periodic wave solutions were given. Recently, the local well-posedness of the FW system (\ref{e12}) in $H^{s}(R)\times H^{s-1}(R)$, $s>\frac{3}{2}$, has been established in \cite{5}. In view of the local well-posedness result, we find that the solution of FW system (\ref{e12}) continuously depends on its initial data in $H^{s}(R)\times H^{s-1}(R)$.

One of aims in this paper is to show the well-posedness is sharp in the sense that the data-to-solution map not uniformly continuous but H\"{o}lder continuous in a weaker topology. Here we mainly use the method of approximate solutions to show the nonuniform continuity of data-to-solution map. As far as we know, a family of high-low frequency approximate solutions was firstly introduced by Koch and Tzvetkov in \cite{1} to prove that data-to-solution map of Benjamin-Ono equation was not better than continuous. Later, more and more high-low frequency approximate solutions were constructed to show that the map of different evolution equations (systems) was not uniformly continuous. For instance, Himonas and Kenig in \cite{2} used the method to CH equation on the line and Wang et al. in \cite{19} and Yang in \cite{7} applied for the rotation-two-component CH system and so on. Unfortunately, these constructions are not suitable for FW system due to the special form of the right hand side of (\ref{e12}). Thus, here we reconstruct a new family of approximate solutions based on the ideas in \cite{2,4,8} to show the nonuniform continuity. In addition, the research of H\"{o}lder continuity of data-to-solution map in a weaker topology has attracted numerous interests. We refer to \cite{20}-\cite{22} for the b-equation and \cite{11} for two-component higher order Camassa-Holm system, which enable us to have a good understanding of the well-posedness problem of evolution equations.

Another motivation for this paper comes from the idea of \cite{3}, where authors solved an open problem left in \cite{23} and they answered positively that CH equation was ill-posedness in critical Sobolev space $H^{\frac{3}{2}}(R)$. This raises an interesting question whether the FW system (\ref{e12}) is well-posed in critical space $H^{\frac{3}{2}}(R)\times H^{\frac{1}{2}}(R)$. Considering the particularity in the structure of FW system, here we construct the special initial data to meet the requirement of norm inflation, which implies the ill-posedness in $H^{\frac{3}{2}}(R)\times H^{\frac{1}{2}}(R)$. Similarly, these results can easily be extended to a series of Besov spaces.

The plan of the paper is as follows. Section 2 is devoted to collect some useful lemmas that we need later. In section 3, we apply the method of approximate solutions to establish that the data-to-solution map of FW system is not uniformly continuous and use the energy method to show that it's H\"{o}lder continuous in a weaker topology. In the last section, we prove that the FW system (\ref{e12}) is ill-posed in critical space $H^{\frac{3}{2}}(R)\times H^{\frac{1}{2}}(R)$.
\section{Preliminaries}
In this section, we collect some useful lemmas that we need and give a more refined priori estimate of solution to (\ref{e12}). Throughout this paper, we denote $f\lesssim g$ when
$f\leq cg$ for some constant $c>0$, and $f\approx g$ when $f\lesssim g \lesssim f$.
\begin{lemma}\label{lem2.1}(see \cite{1})
Let $\phi\in S(R)$, $\delta>0$ and $c\in R$. Then, for any $s\geq 0$, we have that
$$
\lim_{n\rightarrow \infty} n^{-\frac{\delta}{2}-s}\|\phi(\frac{x}{n^{\delta}})\cos(nx-c)\|_{H^{s}(R)}=\frac{1}{\sqrt{2}}\|\phi\|_{L^{2}(R)}.
$$
This relation is also true if the function $\cos$ is replaced by $\sin$.
\end{lemma}
\begin{lemma}\label{lem2.2} (Interpolation lemma)
Let $s_{1}< s < s_{2}$ be real numbers, then
$$
  \|f\|_{H^{s}}\leq \|f\|_{H^{s_{1}}}^{\frac{s_{2}-s}{s_{2}-s_{1}}} \|f\|_{H^{s_{2}}}^{\frac{s-s_{1}}{s_{2}-s_{1}}}.
$$
\end{lemma}
\begin{lemma}\label{lem2.3} (See \cite{10})
If $s>\frac{3}{2}$ and $0\leq k+1\leq s$, then there exists a constant $c>0$ such that
$$
\|[\Lambda^{k}\partial_{x},f]g\|_{L^{2}(R)}\leq c\|f\|_{H^{s}(R)}\|g\|_{H^{k}(R)},
$$
where $[\Lambda^{k}\partial_{x},f]:=\Lambda^{k}\partial_{x}f-f\Lambda^{k}\partial_{x}$.
\end{lemma}
\begin{lemma}\label{lem2.4}
For $s>0$ and $f, g\in H^{s}(R)\cap L^{\infty}(R)$, we have
$$
\|fg\|_{H^{s}(R)}\leq C\|f\|_{L^{\infty}(R)}\|g\|_{H^{s}(R)}+\|f\|_{H^{s}(R)}\|g\|_{L^{\infty}(R)};
$$
For $s>\frac{1}{2}$ and $f, g\in H^{s}(R)$, we have
$$
\|fg\|_{H^{s}(R)}\leq C\|f\|_{H^{s}(R)}\|g\|_{H^{s}(R)}.
$$
\end{lemma}
\begin{lemma}\label{lem2.5} (See \cite{3})
Let $T>0$ could be infinity, assume $A(t)\in C^{1}([0,T))$, $A(t)>0$ and there exists a constant $C>0$ such that
$$
\frac{d}{dt}A(t)\leq C A(t)ln(e+A(t)), ~~for~~t\in[0,T).
$$
Then we have
$$
A(t)\leq (e+A(0))^{e^{Ct}}, ~~for~~t\in[0,T).
$$
\end{lemma}
\begin{lemma}\label{lem2.6}(Littlewood-Paley decomposition)
There exists a couple of smooth radial function $(\chi, \varphi)$ valued in
$[0, 1]$ such that $\chi$ is supported in the ball $B=\{\xi\in R,  |\xi| \leq \frac{4}{3}\}$ and $\varphi$
is supported in the ring $\mathcal{C}= \{\xi\in R, \frac{3}{4}\leq|\xi|\leq \frac{8}{3}\}$. Moreover,
$$
\forall \xi \in R,~~ \chi(\xi)+ \sum_{j\in 0}\varphi(2^{-j}\xi)=1
$$
and
$$
Supp \varphi(2^{-j}\cdot)\cap Supp \varphi(2^{-j'}\cdot)=\emptyset, ~~if~~ |j-j'|\geq 2,
$$
$$
Supp \chi(\cdot)\cap Supp \varphi(2^{-j}\cdot)=\emptyset, ~~if~~ |j|\geq 1.
$$
Then for $u\in \mathcal{S'}$, we have
$$
u=\sum_{j\geq -1}\Delta_{j}u ~~in~~\mathcal{S'}(R),
$$
where the nonhomogeneous dyadic operators are defined by
$$
\Delta_{j}u=0,~~if~~ j\leq -2,
$$
$$
\Delta_{-1}u=\chi(D)u=\mathcal{F}^{-1}(\chi(\xi)\hat{u}(\xi))(x),
$$
$$
\Delta_{j}u=\varphi(2^{-j}D)u=\mathcal{F}^{-1}(\varphi(2^{-j}\xi)\hat{u}(\xi))(x),~~if~~j\geq 0.
$$
\end{lemma}
\begin{definition}\label{def2.7}(Besov spaves)
Let $s\in R$, $1\leq p,r\leq \infty$. The inhomogeneous Besov space $B^{s}_{p,r}(R)$ is defined by
$$
B^{s}_{p,r}(R):=\{f\in S'(R): \|f\|_{B^{s}_{p,r}}<\infty\},
$$
where
$$
\|f\|_{B^{s}_{p,r}}=\left\{ \begin{array}{ll}
(\sum_{j\in Z}2^{jsr}\|\Delta_{j}f\|_{L^{p}}^{r})^{\frac{1}{r}}, & r<\infty, \\
 \sup_{j\in Z}2^{jsr}\|\Delta_{j}f\|_{L^{p}},  & r=\infty .
 \end{array} \right.
$$
If $s=\infty$, then $B^{\infty}_{p,r}=\bigcap_{s\in R}B^{s}_{p,r}$; for $s\in R$, $p=r=2$, then $B^{s}_{2,2}=H^{s}$.
\end{definition}
\begin{lemma}\label{lem2.8} (Gagliardo-Nirenberg inequality, see \cite{12}) For $s>\frac{1}{2}$, there hold
$$
\|f\|_{L^{\infty}}\leq C_{s}(1+\|f\|_{B^{0}_{\infty,\infty}}log(e+\|f\|_{H^{s}}))
$$
and
$$
\|f\|_{B^{\frac{1}{2}}_{2,\infty}\cap L^{\infty}}\leq C_{s}(1+\|f\|_{B^{\frac{1}{2}}_{2,\infty}}log(e+\|f\|_{H^{s}})),
$$
where $C_{s}$ is a constant depending on $s$.
\end{lemma}
\begin{lemma}\label{lem2.9} (see Lemma 2.100 in \cite{15})
 Let $\sigma \in R $, $ 1\leq r \leq \infty$, and $ 1\leq p \leq p_{1}\leq \infty$, let $v$ be a vector field
 over $R^{d}$. Assume that $\sigma > -d \min\{\frac{1}{p_{1}}, 1-\frac{1}{p}\}$,
 define $R_{j}=[v\cdot\nabla, \Delta_{j}]f$, there exists a constant $C$, depending continuously on $p, p_{1},\sigma$ and $d$, such that
 \begin{align}
 \|(2^{j\sigma}\|R_{j})\|_{L^{p}})_{j}\|_{l^{r}}\leq C\|\nabla v\|_{B_{p,\infty}^{\frac{d}{p}}\cap L^{\infty}}\|f\|_{B_{p,r}^{\sigma}}, ~~for~~ \sigma< 1+\frac{d}{p_{1}}. \nonumber
 \end{align}
 Further, if $\sigma > 0$ and $ \frac{1}{p_{2}}=\frac{1}{p}-\frac{1}{p_{1}}$, then
 \begin{align}
 \|(2^{j\sigma}\|R_{j})\|_{L^{p}})_{j}\|_{l^{r}}\leq C\|\nabla v\|_{L^{\infty}}\|f\|_{B_{p,r}^{\sigma}}+\|\nabla f\|_{L^{p_{2}}}\|\nabla v\|_{B_{p_{1},r}^{\sigma-1}}. \nonumber
 \end{align}
\end{lemma}

Different from the energy estimate obtained in \cite{5}, here we obtain a more refined priori estimate for solutions $(u,\eta)$ to FW system (\ref{e12}) with the help of techniques developed in Besov space.
\begin{lemma}\label{lem2.10}
Assume that $(u, \eta)\in H^{s}\times H^{s-1}$ with $\frac{1}{2}< s-1< \frac{3}{2}$ is the smooth solution to $(\ref{e12})$, then
 \begin{equation}\label{e21}
 \frac{d}{dt}(\|u\|_{H^{s}}+\|\eta\|_{H^{s-1}})\leq
 C(\|\partial_{x}u\|_{B_{2,\infty}^{\frac{1}{2}}\cap L^{\infty}}+\|\eta\|_{\infty}+1)(\|u\|_{H^{s}}+\|\eta\|_{H^{s-1}}).
\end{equation}
\end{lemma}
\begin{proof}
 Applying the localization operator to $(\ref{e12})$, we transform the FW system into the following system
  \begin{equation}\label{e22}
 \left\{ \begin{array}{ll}
 \partial_{t}\Delta_{j}u+u\partial_{x}\Delta_{j}u=[u\partial_{x}, \Delta_{j}]u+\partial_{x}\Lambda^{-2}(\Delta_{j}\eta-\Delta_{j}u), \\
 \partial_{t}\Delta_{j}\eta+u\partial_{x}\Delta_{j}\eta = [u\partial_{x}, \Delta_{j}]\eta-\Delta_{q}(\eta\partial_{x}u), \\
 \Delta_{j}u\mid_{t=0}= \Delta_{j}u_{0}, ~~\Delta_{j}\eta\mid_{t=0}=\Delta_{j}\eta_{0},
 \end{array} \right.
 \end{equation}
along the flow of $u$. Multiplying both sides of the second equation in (\ref{e22}) by $\Delta_{j}\eta$, integrating over $R$ with respect to $x$ and using the Lemma \ref{lem2.9}, we have
\begin{equation}
\begin{aligned}
 \frac{1}{2}\frac{d}{dt}\|\Delta_{j}\eta\|_{2}^{2}&=-\int_{R}u\partial_{x}\Delta_{j}\eta\Delta_{j}\eta dx
 +\int_{R}[u\partial_{x}, \Delta_{j}]\eta \Delta_{j}\eta dx -\int_{R}\Delta_{j}(\eta\partial_{x}u)\Delta_{j}\eta dx  \\ \nonumber
 & \leq\|\partial_{x}u\|_{\infty}\|\Delta_{j}\eta\|_{2}^{2}+\|[u\partial_{x}, \Delta_{j}]\eta\|_{2}\|\Delta_{j}\eta\|_{2}+\|\Delta_{j}(\eta\partial_{x}u)\|_{2}\|\Delta_{j}\eta\|_{2}\\ \nonumber &\leq\|\partial_{x}u\|_{\infty}\|\Delta_{j}\eta\|_{2}^{2}+ 2^{-j(s-1)}c_{j}\|\eta\|_{H^{s-1}}\|\partial_{x}u\|_{B_{2,\infty}^{\frac{1}{2}}\cap L^{\infty}}\|\Delta_{j}\eta\|_{2} \\ \nonumber
 &+\|\Delta_{j}(\eta\partial_{x}u)\|_{2}\|\Delta_{j}\eta\|_{2},
 \end{aligned}
\end{equation}
where $c_{j}\in l^{2}$, which implies
 \begin{equation}\label{e23}
 \frac{d}{dt}\|\Delta_{j}\eta\|_{2}\leq\|\partial_{x}u\|_{\infty}\|\Delta_{j}\eta\|_{2}+ 2^{-j(s-1)}c_{j}\|\eta\|_{H^{s-1}}\|\partial_{x}u\|_{B_{2,\infty}^{\frac{1}{2}}\cap L^{\infty}}+
 \|\Delta_{j}(\eta\partial_{x}u)\|_{2}
 \end{equation}
Multiplying (\ref{e23}) by $2^{j(s-1)}$ and taking the $l^{2}$ norm over $j$, we obtain
 \begin{align}\label{e24}
 \frac{d}{dt}\|\eta\|_{H^{s-1}}&\leq\|\partial_{x}u\|_{\infty}\|\eta\|_{H^{s-1}}+
 \|\eta\partial_{x}u\|_{H^{s-1}} \\ \nonumber
 &\leq C(\|\partial_{x}u\|_{B_{2,\infty}^{\frac{1}{2}}\cap L^{\infty}}+\|\eta\|_{\infty})(\|\eta\|_{H^{s-1}}+\|u\|_{H^{s}}),
 \end{align}
where we use the Lemma \ref{lem2.4}. Similar process carried out on the first equation in (\ref{e22}), we get
 \begin{equation}\label{e25}
 \frac{d}{dt}\|u\|_{H^{s}}\leq C\|\partial_{x}u\|_{\infty}\|u\|_{H^{s}}+\|\eta\|_{H^{s-1}}+\|u\|_{H^{s}}.
 \end{equation}
 Adding (\ref{e24}) to (\ref{e25}), we attain
$$
 \frac{d}{dt}(\|\eta\|_{H^{s-1}}+\|u\|_{H^{s}})\leq C(\|\partial_{x}u\|_{B_{2,\infty}^{\frac{1}{2}}\cap L^{\infty}}+\|\eta\|_{\infty}+1)(\|\eta\|_{H^{s-1}}+\|u\|_{H^{s}}).
$$
\end{proof}
\begin{remark}\label{rem2.11}
Define the energy
$$
y(t)=E_{s}(u,\eta)=:\|u(t)\|_{H^{s}}+\|\eta(t)\|_{H^{s-1}}
$$
with $y(0)=\|u_{0}\|_{H^{s}}+\|\eta_{0}\|_{H^{s-1}}$ with $s>\frac{3}{2}$. Based on the local well-posedness result in \cite{5}, we have
$$
\frac{dy(t)}{dt}\leq C(y(t)+y^{2}(t)),
$$
which implies that
$$
-\frac{d(y^{-1}+1)}{dt}\leq C(y^{-1}+1).
$$
Then we can obtain
$$
y\leq \frac{1}{e^{-Ct}(y_{0}^{-1}+1)-1}.
$$
Let $\widetilde{T}:=\frac{1}{2C}ln(1+\frac{1}{y_{0}})$, then the solution $(u,\eta)$ exists for $t\in[0,\widetilde{T}]$ and there holds
$$
y(t)\leq \frac{1}{e^{-C\widetilde{T}}(y_{0}^{-1}+1)-1}\leq 2e^{C\widetilde{T}}y_{0},
$$
that is to say,
\begin{equation}\label{e26}
\|u(t)\|_{H^{s}}+\|\eta(t)\|_{H^{s-1}}\leq 2e^{C\widetilde{T}}(\|u_{0}\|_{H^{s}}+\|\eta_{0}\|_{H^{s-1}}).
\end{equation}
\end{remark}
\section{Continuity properties of the data-to-solution map}

\subsection{Nonuniform continuity}
In view of the local well-posedness result in \cite{5}, it's know that the solution $(u,\eta)$ of FW system continuously relies on its its
data $(u_{0},\eta_{0})$ in $H^{s}(R)\times H^{s-1}(R)$ with $s>\frac{3}{2}$. In this subsection, we aim to establish that the dependence on the
initial data is sharp.
\begin{theorem}\label{thm3.1}(Non-uniform continuity of data-to-solution map)
If $s>\frac{3}{2}$, the data-to-solution map $(u_{0},\eta_{0})\mapsto (u(t), \eta(t))$ for the Cauchy problem of FW system (\ref{e12}) is not uniformly
continuous from any bounded subset of $H^{s}(R)\times H^{s-1}(R)$ into $C([0,T); H^{s}(R)\times H^{s-1}(R))$.
\end{theorem}

Here we would employ the method of approximate solutions introduced in \cite{1,2}. The key idea of the method is to show that there exists a two-parameter
family of actual solutions $(u_{\alpha,n}(x,t),\eta_{\alpha,n}(x,t))\in C([0,T);H^{s}(R)\times H^{s-1}(R))$ with $\alpha=0,1$ and $n\geq 1$ such that
\begin{equation}\label{e31}
\lim_{n\rightarrow\infty}(\|u_{1,n}(t)\|_{H^{s}}+\|u_{0,n}(t)\|_{H^{s}}+\|\eta_{1,n}(t)\|_{H^{s-1}}+\|\eta_{0,n}(t)\|_{H^{s-1}})\leq C,
\end{equation}
\begin{equation}\label{e32}
\lim_{n\rightarrow\infty}(\|u_{1,n}(0)-u_{0,n}(0)\|_{H^{s}}+\|\eta_{1,n}(0)-\eta_{0,n}(0)\|_{H^{s-1}})\rightarrow 0,
\end{equation}
\begin{equation}\label{e33}
\lim_{n\rightarrow\infty}(\|u_{1,n}(t)-u_{0,n}(t)\|_{H^{s}}+\|\eta_{1,n}(t)-\eta_{0,n}(t)\|_{H^{s-1}})\geq 2|\sin\frac{t}{2}|
\end{equation}
hold for all $t<T$, where $T$ is the lifespan of solutions. To achieve there aims, we divide the proof into following two steps. Namely, in the first step we will construct the approximate solutions
$(u^{\alpha,n}(x,t),\eta^{\alpha,n}(x,t))$ and show that the approximate solutions are indeed approximations to the actual solutions. In the second step,
we shall establish (\ref{e31})-(\ref{e33}) by using the properties of approximate solutions.
\begin{proof}
{\bf Step 1:}
Inspired by \cite{2}, we first construct two two-parameter approximate solutions $(u^{\alpha, n}(x,t), \eta^{\alpha,n}(x,t))$ with $\alpha=0,1$ and $n\geq
1$ by
\begin{equation}\label{e34}
u^{\alpha,n}(x,t)=\frac{\alpha}{n}\psi(\frac{x}{n^{\delta}})+n^{-s-\frac{\delta}{2}}\phi(\frac{x}{n^{\delta}})\cos(nx-\alpha
t),~~\eta^{\alpha,n}(x,t)=\frac{\alpha}{n}\psi(\frac{x}{n^{\delta}}),
\end{equation}
where $\phi, \psi\in C^{\infty}_{0}(R)$ are two cut-off function satisfying
\begin{equation}\label{e35}
\phi(x)=\left\{ \begin{array}{ll}
 1, & |x|<1, \\
 0,  & |x|\geq 2.
\end{array} \right.~~~
\psi(x)=\left\{ \begin{array}{ll}
 1, & |x|<2, \\
 0,  & |x|\geq 3.
\end{array} \right.
\end{equation}
From Lemma \ref{lem2.1}, we have that for any $r\geq 0$
\begin{equation}\label{e36}
\|\phi(\frac{x}{n^{\delta}})\cos(nx-\alpha t)\|_{H^{r}(R)}\approx n^{\frac{\delta}{2}+r},~~\|\phi(\frac{x}{n^{\delta}})\sin(nx-\alpha
t)\|_{H^{r}(R)}\approx n^{\frac{\delta}{2}+r}.
\end{equation}
In addition, for any $r\geq 0$ and $\widetilde{\phi}\in C^{\infty}_{0}(R)$, it's easy to see
\begin{equation}\label{e37}
\|\widetilde{\phi}(\frac{x}{n^{\delta}})\|_{H^{r}(R)}\leq n^{\frac{\delta}{2}}\|\widetilde{\phi}\|_{H^{r}(R)}. \end{equation}

Now let's estimate the errors and show the approximate solutions are indeed approximations to the actual solutions. Substituting $(u^{\alpha,n}(x,t),
\eta^{\alpha,n}(x,t))$ into the FW system (\ref{e12}), we get the following errors for the approximate solutions
\begin{equation}\label{e38}
E=\partial_{t} u^{\alpha,n}+u^{\alpha,n}\partial_{x}u^{\alpha,n}-\partial_{x}\Lambda^{-2}(\eta^{\alpha,n}-u^{\alpha,n}):=E_{1}-E_{2}
\end{equation}
and
\begin{equation}\label{e39}
F=\partial_{t} \eta^{\alpha,n}+\partial_{x}(\eta^{\alpha,n}u^{\alpha,n}).
\end{equation}
Moreover, we can obtain the following error estimates.
\begin{lemma}\label{lem3.2}
Assume $s>\frac{3}{2}$ and $\frac{1}{2}<\delta<1$, then there exists $s_{1}\leq s-1$ and $\varepsilon>0$ such that
\begin{equation}\label{e310}
\|E\|_{H^{s_{1}}}\lesssim n^{-s-\varepsilon+s_{1}}, \quad \|F\|_{H^{s_{1}-1}}\lesssim n^{-s-\varepsilon+s_{1}}~~for ~~0\leq t<T.
\end{equation}
\end{lemma}
\begin{proof}
From (\ref{e34}), (\ref{e35}) and the properties of trigonometric functions, we have
\begin{align}\label{e311}
E_{1}
 &=\partial_{t} u^{\alpha,n}+u^{\alpha,n}\partial_{x}u^{\alpha,n}\nonumber\\
 &=\alpha n^{-s-\frac{\delta}{2}}\phi(\frac{x}{n^{\delta}})\sin(nx-\alpha t) \nonumber\\
 &+\alpha^{2}n^{-2-\alpha}\psi(\frac{x}{n^{\delta}})\partial_{x}\psi(\frac{x}{n^{\delta}})+\alpha
 n^{-s-1-\frac{3\delta}{2}}\partial_{x}\phi(\frac{x}{n^{\delta}})\psi(\frac{x}{n^{\delta}})\cos(nx-\alpha t) \nonumber\\
 &-\alpha n^{-s-\frac{\delta}{2}}\psi(\frac{x}{n^{\delta}})\phi(\frac{x}{n^{\delta}})\sin(nx-\alpha t)+\alpha
 n^{-s-1-\frac{3\delta}{2}}\partial_{x}\psi(\frac{x}{n^{\delta}})\phi(\frac{x}{n^{\delta}})\cos(nx-\alpha t) \nonumber\\
 &+n^{-2s-2\delta}\phi(\frac{x}{n^{\delta}})\partial_{x}\phi(\frac{x}{n^{\delta}})\cos^{2}(nx-\alpha
 t)-n^{-2s-\delta+1}\phi^{2}(\frac{x}{n^{\delta}})\sin(nx-\alpha t)\cos(nx-\alpha t)\nonumber\\
 &=\alpha^{2}n^{-2-\alpha}\psi(\frac{x}{n^{\delta}})\partial_{x}\psi(\frac{x}{n^{\delta}})+\alpha
 n^{-s-1-\frac{3\delta}{2}}\partial_{x}\phi(\frac{x}{n^{\delta}})\psi(\frac{x}{n^{\delta}})\cos(nx-\alpha t) \nonumber\\
 &+n^{-2s-2\delta}\phi(\frac{x}{n^{\delta}})\partial_{x}\phi(\frac{x}{n^{\delta}})\cos^{2}(nx-\alpha
 t)-n^{-2s-\delta+1}\phi^{2}(\frac{x}{n^{\delta}})\sin(nx-\alpha t)\cos(nx-\alpha t) \nonumber\\
 &=\alpha^{2}n^{-2-\alpha}\psi(\frac{x}{n^{\delta}})\partial_{x}\psi(\frac{x}{n^{\delta}})+\alpha
 n^{-s-1-\frac{3\delta}{2}}\partial_{x}\phi(\frac{x}{n^{\delta}})\psi(\frac{x}{n^{\delta}})\cos(nx-\alpha t) \nonumber\\
 &+\frac{1}{2}n^{-2s-2\delta}\phi(\frac{x}{n^{\delta}})\partial_{x}\phi(\frac{x}{n^{\delta}})\cos(2nx-2\alpha t)+
 \frac{1}{2}n^{-2s-2\delta}\phi(\frac{x}{n^{\delta}})\partial_{x}\phi(\frac{x}{n^{\delta}}) \nonumber\\
 &-\frac{1}{2}n^{-2s-\delta+1}\phi^{2}(\frac{x}{n^{\delta}})\sin(2nx-2\alpha t).
\end{align}
Thus, (\ref{e36}), (\ref{e37}) and (\ref{e311}) yield
\begin{equation}\label{e312}
\|E_{1}\|_{H^{s_{1}}}\lesssim n^{-2}+n^{-s-1-\frac{\delta}{2}+s_{1}}+n^{-s-s-\delta+s_{1}}+n^{-2s-\delta}+n^{-s-s+1+s_{1}}.
\end{equation}
In addition, it's obvious that
\begin{align}\label{e313}
\|-E_{2}\|_{H^{s_{1}}}
 &=\|\partial_{x}\Lambda^{-2}(-\eta^{\alpha,n}+u^{\alpha,n})\|_{H^{s_{1}}} \nonumber\\
 &\leq \frac{2}{n^{1+\delta}}\|\Lambda^{-2}\partial_{x}\psi(\frac{x}{n^{\delta}})\|_{H^{s_{1}}}+n^{-s-\frac{\delta}{2}}\|\partial_{x}\Lambda^{-2}
 (\phi(\frac{x}{n^{\delta}})cos(nx-\alpha t))\|_{H^{s_{1}}} \nonumber\\
 &\lesssim n^{-1-\delta}\|\partial_{x}\psi(\frac{x}{n^{\delta}})\|_{H^{s_{1}-2}}+n^{-s-\frac{\delta}{2}}\|\phi(\frac{x}{n^{\delta}})cos(nx-\alpha t)\|_{H^{s_{1}-1}} \nonumber\\
 &\leq n^{-1-\delta}(\int_{R}|\xi|^{2}(1+|\xi|^{2})^{s_{1}-2}n^{2\delta}|\hat{\psi}(n^{\delta}\xi)|^{2}d\xi)^{\frac{1}{2}}+n^{-s-1+s_{1}} \nonumber\\
 &= n^{-1-\frac{3\delta}{2}}(\int_{R}|y|^{2}(1+n^{-2\delta}|y|^{2})^{s_{1}-2}|\hat{\psi}(y)|^{2}dy)^{\frac{1}{2}}+n^{-s-1+s_{1}} \nonumber\\
 &\leq n^{-1-\frac{3\delta}{2}}(\int_{R}|y|^{2}(1+|y|^{2})^{s_{1}-2}|\hat{\psi}(y)|^{2}dy)^{\frac{1}{2}}+n^{-s-1+s_{1}} \nonumber\\
 &=n^{-1-\frac{3\delta}{2}}\|\partial_{x}\psi\|_{H^{s_{1}-2}}+n^{-s-1+s_{1}} \nonumber\\
 &\lesssim n^{-1-\frac{3\delta}{2}}+n^{-s-1+s_{1}},
\end{align}
where we use (\ref{e36}), (\ref{e37}) and the fact $\frac{1}{2}<\delta<1$. From (\ref{e34})-(\ref{e37}), we also have
\begin{align}\label{e314}
\|F\|_{H^{s_{1}-1}}
 &=\|\partial_{x}(\alpha^{2}n^{-2}\psi^{2}(\frac{x}{n^{\delta}})+\alpha n^{-s-\frac{\delta}{2}-1}\psi(\frac{x}{n^{\delta}})\cos(nx-\alpha t)\phi(\frac{x}{n^{\delta}}))\|_{H^{s_{1}-1}} \nonumber\\
 &\leq \|\alpha^{2}n^{-2}\psi^{2}(\frac{x}{n^{\delta}})\|_{H^{s_{1}}}+\|\alpha n^{-s-\frac{\delta}{2}-1}\cos(nx-\alpha t)\phi(\frac{x}{n^{\delta}})\|_{H^{s_{1}}} \nonumber\\
&\lesssim n^{-2+\frac{\delta}{2}}+n^{-s-1+s_{1}}.
\end{align}
Thus (\ref{e310}) follows from (\ref{e312}), (\ref{e313}) and (\ref{e314}).
\end{proof}
{\bf Step 2:} Now we are ready to show (\ref{e31})-(\ref{e33}). Let's first give two sequences of solutions
$(u_{\alpha,n}(x,t),\eta_{\alpha,n}(x,t))$, where $\alpha=\{0,1\}$, to the FW system (\ref{e12}) with initial data
\begin{equation}\label{e315}
 \begin{array}{ll}
 u_{\alpha,n}(x,0)=u^{\alpha,n}(x,0)=\frac{\alpha}{n}\psi(\frac{x}{n^{\delta}})+n^{-s-\frac{\delta}{2}}\phi(\frac{x}{n^{\delta}})\cos(nx),\\
 \eta_{\alpha,n}(x,0)=\eta^{\alpha,n}(x,0)=\frac{\alpha}{n}\psi(\frac{x}{n^{\delta}}).
\end{array}
\end{equation}
From the local well-posedness result, it's known that $(u_{\alpha,n}(x,t),\eta_{\alpha,n}(x,t))\in C([0,T);H^{s}\times H^{s-1})$ for
$s>\frac{3}{2}$. The energy estimate in (\ref{e26}) and (\ref{e315}) imply
\begin{align}\label{e316}
~
 &\|u_{1,n}(t)\|_{H^{s}}+\|u_{0,n}(t)\|_{H^{s}}+\|\eta_{1,n}(t)\|_{H^{s-1}}+\|\eta_{0,n}(t)\|_{H^{s-1}} \nonumber\\
 &\leq 2e^{C\widetilde{T}}(\|u_{1,n}(0)\|_{H^{s}}+\|u_{0,n}(0)\|_{H^{s}}+\|\eta_{1,n}(0)\|_{H^{s-1}}+\|\eta_{0,n}(0)\|_{H^{s-1}})
 \nonumber\\
 &\lesssim 2e^{C\widetilde{T}}(2n^{-1+\frac{\delta}{2}}+2).
 \end{align}
Then (\ref{e31}) follows from (\ref{e316}) by letting $n\rightarrow\infty$.
In view of (\ref{e315}) and the fact $\frac{1}{2}<\delta<1$, it's easy to see that (\ref{e32}) holds by
\begin{align}\label{e317}
~
 &\lim_{n\rightarrow\infty}(\|u_{1,n}(0)-u_{0,n}(0)\|_{H^{s}}+\|\eta_{1,n}(0)-\eta_{0,n}(0)\|_{H^{s-1}}) \nonumber\\
 &=\lim_{n\rightarrow\infty}(\|n^{-1}\psi(\frac{x}{n^{\delta}})\|_{H^{s}}+\|n^{-1}\psi(\frac{x}{n^{\delta}})\|_{H^{s-1}})
 \nonumber\\
 &\leq\lim_{n\rightarrow\infty} n^{-1+\frac{\delta}{2}}(\|\psi\|_{H^{s}}+\|\psi\|_{H^{s-1}})\rightarrow 0.
 \end{align}

At last, it remains to establish (\ref{e33}). Define the difference between approximate solutions $(u^{\alpha,n}(x,t),\eta^{\alpha,n}(x,t))$ and actual
solutions $(u_{\alpha,n}(x,t),\eta_{\alpha,n}(x,t))$ by
\begin{equation}\label{e318}
\omega_{\alpha}:=u^{\alpha,n}(x,t)-u_{\alpha,n}(x,t),\quad \rho_{\alpha}:=\eta^{\alpha,n}(x,t)-\eta_{\alpha,n}(x,t).
\end{equation}
For $s>\frac{3}{2}$ and $0\leq t<T$, let's first verify that
 \begin{equation}\label{e319}
\|\omega_{\alpha}\|_{H^{s}}+\|\rho_{\alpha}\|_{H^{s-1}}\rightarrow 0 \quad as ~~n\rightarrow\infty.
\end{equation}
Here we use the interpolation idea as in \cite{6}. For $s_{1}<s<s_{2}$, we first establish the energy estimate in $H^{s_{1}}$ norm, then estimate
$H^{s_{2}}$ norm and finally we can obtain the estimate in $H^{s}$.
It's obvious that $(\omega_{\alpha},\rho_{\alpha})$ satisfy
 \begin{equation}\label{e320}
 \left\{ \begin{array}{ll}
 \partial_{t}\omega_{\alpha}+\frac{1}{2}\partial_{x}(\omega_{\alpha}(u^{\alpha,n}+u_{\alpha,n}))-\partial_{x}\Lambda^{-2}(\rho_{\alpha}-\omega_{\alpha})=E,
 \\
 \partial_{t}\rho_{\alpha}+\frac{1}{2}\partial_{x}(\omega_{\alpha}(\eta^{\alpha,n}+\eta_{\alpha,n})+\rho_{\alpha}(u^{\alpha,n}+u_{\alpha,n}))=F,\\
 \omega_{\alpha}(x,0)=0, ~~\rho_{\alpha}(x,0)=0.
 \end{array} \right.
 \end{equation}
Note that the error bounds in $H^{s_{1}}\times H^{s_{1}-1}$ can be seen in the following lemma.
\begin{lemma}\label{lem3.3}
Assume $s>\frac{3}{2}$ and $\frac{1}{2}<\delta<1$, then there exists $s_{1}\leq s-1$ and $\varepsilon>0$ such that
\begin{equation}\label{e321}
 \|\omega_{\alpha}(t)\|_{H^{s_{1}}}+\|\rho_{\alpha}(t)\|_{H^{s_{1}-1}}\lesssim n^{-s-\varepsilon+s_{1}},~~for ~~0\leq t< T.
\end{equation}
\end{lemma}
\begin{proof}
Applying $\Lambda^{s_{1}}$ to the first formula in (\ref{e320}), multiplying both sides by $\Lambda^{s_{1}}\omega_{\alpha}$  and integrating on $R$, we obtain
\begin{align}\label{e322}
\frac{d}{dt}\|\omega_{\alpha}\|^{2}_{H^{s_{1}}}
 &=2\langle\Lambda^{s_{1}}E,\Lambda^{s_{1}}\omega_{\alpha}\rangle-
\langle\partial_{x}\Lambda^{s_{1}}(\omega_{\alpha}(u^{\alpha,n}+u_{\alpha,n})),\Lambda^{s_{1}}\omega_{\alpha}\rangle \nonumber\\
 &+2\langle\partial_{x}\Lambda^{s_{1}-2}(\rho_{\alpha}-\omega_{\alpha}), \Lambda^{s_{1}}\omega_{\alpha}\rangle.
 \end{align}
By H\"{o}lder inequality, we have
$$
|\langle\Lambda^{s_{1}}E,\Lambda^{s_{1}}\omega_{\alpha}\rangle|\leq \|\Lambda^{s_{1}}E\|_{L^{2}}\|\Lambda^{s_{1}}\omega_{\alpha}\|_{L^{2}}=\|E\|_{H^{s_{1}}}\|\omega_{\alpha}\|_{H^{s_{1}}}.
$$
From H\"{o}lder inequality and Lemma \ref{lem2.3}, we get
\begin{equation}
\begin{aligned}
~
 &|\langle\partial_{x}\Lambda^{s_{1}}(\omega_{\alpha}(u^{\alpha,n}+u_{\alpha,n})),\Lambda^{s_{1}}\omega_{\alpha}\rangle| \nonumber\\
 &=|\langle[\partial_{x}\Lambda^{s_{1}},u^{\alpha,n}+u_{\alpha,n}]\omega_{\alpha},\Lambda^{s_{1}}\omega_{\alpha}\rangle+\langle (u^{\alpha,n}+u_{\alpha,n})\partial_{x}\Lambda^{s_{1}}\omega_{\alpha},\Lambda^{s_{1}}\omega_{\alpha} \rangle|
 \nonumber\\
 &\leq|\langle[\partial_{x}\Lambda^{s_{1}},u^{\alpha,n}+u_{\alpha,n}]\omega_{\alpha},\Lambda^{s_{1}}\omega_{\alpha}\rangle|+|\langle (u^{\alpha,n}+u_{\alpha,n})\partial_{x}\Lambda^{s_{1}}\omega_{\alpha},\Lambda^{s_{1}}\omega_{\alpha} \rangle|\nonumber\\
 &\leq \|[\partial_{x}\Lambda^{s_{1}},u^{\alpha,n}+u_{\alpha,n}]\omega_{\alpha}\|_{L^{2}}\|\Lambda^{s_{1}}\omega_{\alpha}\|_{L^{2}}+\frac{1}{2}|\langle \partial_{x}(u^{\alpha,n}+u_{\alpha,n}), (\Lambda^{s_{1}}\omega_{\alpha})^{2} \rangle|\nonumber\\
 &\lesssim \|u^{\alpha,n}+u_{\alpha,n}\|_{H^{s}}\|\omega_{\alpha}\|_{H^{s_{1}}}^{2}+\|\partial_{x}(u^{\alpha,n}+u_{\alpha,n})\|_{L^{\infty}}\|\omega_{\alpha}\|_{H^{s_{1}}}^{2}\nonumber\\
 &\lesssim \|\omega_{\alpha}\|_{H^{s_{1}}}^{2}.
\end{aligned}
\end{equation}
In addition, we also have
\begin{equation}
\begin{aligned}
~
 &|\langle\partial_{x}\Lambda^{s_{1}-2}(\rho_{\alpha}-\omega_{\alpha}), \Lambda^{s_{1}}\omega_{\alpha}\rangle| \nonumber\\
 &=\|\partial_{x}\Lambda^{s_{1}-2}\rho_{\alpha}\|_{L^{2}}\|\Lambda^{s_{1}}\omega_{\alpha}\|_{L^{2}}
 +\|\partial_{x}\Lambda^{s_{1}-2}\omega_{\alpha}\|_{L^{2}}\|\Lambda^{s_{1}}\omega_{\alpha}\|_{L^{2}}\nonumber\\
&\leq \|\rho_{\alpha}\|_{H^{s_{1}-1}} \|\omega_{\alpha}\|_{H^{s_{1}}}+\|\omega_{\alpha}\|_{H^{s_{1}}}^{2}
\end{aligned}
\end{equation}
Thus, (\ref{e322}) and estimates obtained above imply
\begin{equation}\label{e323}
\frac{d}{dt}\|\omega_{\alpha}\|_{H^{s_{1}}}\lesssim \|E\|_{H^{s_{1}}}+\|\rho\|_{H^{s_{1}-1}}+\|\omega_{\alpha}\|_{H^{s_{1}}}.
\end{equation}

Similarly, applying $\Lambda^{s_{1}-1}$ to the second formula in (\ref{e320}), multiplying both sides by $\Lambda^{s_{1}-1}\rho_{\alpha}$ and integrating on $R$, we obtain
\begin{align}\label{e324}
\frac{d}{dt}\|\rho_{\alpha}\|^{2}_{H^{s_{1}-1}}
 &=2\langle \Lambda^{s_{1}-1}F, \Lambda^{s_{1}-1}\rho_{\alpha} \rangle \nonumber\\
 &-\langle\partial_{x}\Lambda^{s_{1}-1}(\omega_{\alpha}(\eta^{\alpha,n}+\eta_{\alpha,n})+\rho_{\alpha}(u^{\alpha,n}+u_{\alpha,n})), \Lambda^{s_{1}-1}\rho_{\alpha} \rangle.
 \end{align}
Besides, we also have
\begin{equation}\label{e325}
\frac{d}{dt}\|\rho_{\alpha}\|_{H^{s_{1}-1}}\lesssim \|F\|_{H^{s_{1}-1}}+\|\rho_{\alpha}\|_{H^{s_{1}-1}}+\|\omega_{\alpha}\|_{H^{s_{1}}}.
\end{equation}
Hence the proof of the lemma is finished by using (\ref{e310}), (\ref{e323}), (\ref{e325}) and Gronwall's inequality.
\end{proof}

On the other hand, for $s_{2}>s>\frac{3}{2}$ and $\frac{1}{2}<\delta<1$, we can use (\ref{e26}) and (\ref{e315}) to obtain the error bounds in $H^{s_{2}}(R)\times H^{s_{2}-1}(R)$ by
\begin{align}\label{e326}
~
 &\|\omega_{\alpha}(t)\|_{H^{s_{2}}}+\|\rho_{\alpha}(t)\|_{H^{s_{2}-1}}\nonumber\\
 &=\|u^{\alpha,n}-u_{\alpha,n}\|_{H^{s_{2}}}+\|\eta^{\alpha,n}-\eta_{\alpha,n}\|_{H^{s_{2}-1}} \nonumber\\
 &\leq \|u^{\alpha,n}\|_{H^{s_{2}}}+\|\eta^{\alpha,n}\|_{H^{s_{2}-1}}+\|u_{\alpha,n}\|_{H^{s_{2}}}+\|\eta_{\alpha,n}\|_{H^{s_{2}-1}}\nonumber\\
 &\lesssim n^{s_{2}-s}+n^{\frac{\delta}{2}-1}\lesssim n^{s_{2}-s}.
 \end{align}
Therefore, (\ref{e319}) follows from (\ref{e321}), (\ref{e326}) and Lemma \ref{lem2.2}, that is
$$
\|\omega_{\alpha}\|_{H^{s}}+\|\rho_{\alpha}\|_{H^{s-1}}\lesssim
(n^{-s-\varepsilon+s_{1}})^{\frac{s_{2}-s}{s_{2}-s_{1}}}(n^{s_{2}-s})^{\frac{s-s_{1}}{s_{2}-s_{1}}}=
n^{-\varepsilon\frac{(s_{2}-s)}{s_{2}-s_{1}}}\rightarrow 0,~~as~~n\rightarrow\infty.
$$
Based on (\ref{e34}), (\ref{e36}), (\ref{e319}) and the triangle inequality, we can prove (\ref{e33}) by
\begin{align}\label{e327}
~
 &\lim_{n\rightarrow\infty}(\|u_{1,n}(t)-u_{0,n}(t)\|_{H^{s}}+\|\eta_{1,n}(t)-\eta_{0,n}(t)\|_{H^{s-1}})\nonumber\\
 &\geq \lim_{n\rightarrow\infty}(\|u^{1,n}(t)-u^{0,n}(t)\|_{H^{s}}-\|u_{1,n}(t)-u^{1,n}(t)\|_{H^{s}}-\|u_{0,n}(t)-u^{0,n}(t)\|_{H^{s}})\nonumber\\
 &+\lim_{n\rightarrow\infty}(\|\eta^{1,n}(t)-\eta^{0,n}(t)\|_{H^{s-1}}-\|\eta_{1,n}(t)-\eta^{1,n}(t)\|_{H^{s-1}}-\|\eta_{0,n}(t)-\eta^{0,n}(t)\|_{H^{s-1}})\nonumber\\
 &=\lim_{n\rightarrow\infty}(\|u^{1,n}(t)-u^{0,n}(t)\|_{H^{s}}+\|\eta^{1,n}(t)-\eta^{0,n}(t)\|_{H^{s-1}}) \nonumber\\
 &=\lim_{n\rightarrow\infty}\|\frac{1}{n}\psi(\frac{x}{n^{\delta}})+n^{-s-\frac{\delta}{2}}\phi(\frac{x}{n^{\delta}})(\cos(nx-t)-\cos(nx))\|_{H^{s}} \nonumber\\
 &\geq\lim_{n\rightarrow\infty}\|n^{-s-\frac{\delta}{2}}\phi(\frac{x}{n^{\delta}})(\cos(nx-t)-\cos(nx))\|_{H^{s}}-\lim_{n\rightarrow\infty}n^{-1+\frac{\delta}{2}}\nonumber\\
 &=\lim_{n\rightarrow\infty}\|2n^{-s-\frac{\delta}{2}}\phi(\frac{x}{n^{\delta}})\sin(nx-\frac{t}{2})\sin(\frac{t}{2})\|_{H^{s}}\geq 2|\sin(\frac{t}{2})|.
 \end{align}
 Up to now, the non-uniform continuity of data-to-solution map is completed.
\end{proof}

\subsection{H\"{o}lder continuity}
Although the date-to-uniform map is not uniformly continuous in $H^{s}(R)\times H^{s-1}(R)$ for $s>\frac{3}{2}$, we are able to prove that the map is H\"{o}lder continuous if choosing a properly weakened topology, which can be summarized in following theorem.
\begin{theorem}\label{thm3.4}
Assume $s>\frac{3}{2}$ and $s-1\leq r<s$, then the solution map for FW system (\ref{e12}) is H\"{o}lder continuous with H\"{o}lder exponent $\beta=s-r$ as a map from set $Q_{m}=\{(u,\eta)\in H^{s}(R)\times H^{s-1}(R): \|u\|_{H^{s}}+\|\eta\|_{H^{s-1}}\leq m\}$ with $H^{r}(R)\times H^{r-1}(R)$ norm to $C([0,\widetilde{T}];H^{r}(R)\times H^{r-1}(R))$, namely,
$$
\|(u(t),\eta(t))-(v(t),\theta(t))\|_{C([0,\widetilde{T}];H^{r}\times H^{r-1})}\leq C\|(u_{0},\eta_{0})-(v_{0},\theta_{0})\|^{\beta}_{H^{r}\times H^{r-1}},
$$
where $C$ depends on $s,r,\widetilde{T},m$ and $(u(t),\eta(t)), (v(t),\theta(t))$ are two solutions for (\ref{e12}) corresponding to the initial data $(u_{0},\eta_{0}),(v_{0},\theta_{0})\in Q_{m}$, respectively.
\end{theorem}
\begin{proof}
Since $(u,\eta)\in H^{s}(R)\times H^{s-1}(R)$ and $(v,\theta)\in H^{s}(R)\times H^{s-1}(R)$ are solutions to the FW system (\ref{e12}), if define
$$
\omega=u-v,\quad \rho=\eta-\theta,
$$
then it's easy to see $(\omega,\rho)$ satisfy
\begin{equation}\label{e328}
 \left\{ \begin{array}{ll}
 \omega_{t}+\frac{1}{2}(\omega(u+v))=\partial_{x}\Lambda^{-2}(\rho-\omega),
 \\
 \rho_{t}+(u\rho+\theta \omega)_{x}=0,\\
 \omega(x,0)=u_{0}-v_{0}, ~~\rho(x,0)=\eta_{0}-\theta_{0}.
 \end{array} \right.
\end{equation}
From the interpolation Lemma \ref{lem2.2}, we have
\begin{align}\label{e329}
 \|\omega\|_{H^{r}(R)}+\|\rho\|_{H^{r-1}(R)}
 &\leq \|\omega\|_{H^{s-1}}^{s-r}\|\omega\|_{H^{s}}^{1+r-s}+\|\rho\|_{H^{s-2}}^{s-r}\|\rho\|_{H^{s-1}}^{1+r-s}  \nonumber\\
 &\leq (\|\omega\|_{H^{s}}^{1+r-s}+\|\rho\|_{H^{s-1}}^{1+r-s})(\|\omega\|_{H^{s-1}}^{s-r}+\|\rho\|_{H^{s-2}}^{s-r}) \nonumber\\
 &\leq 4(\|\omega\|_{H^{s}}+\|\rho\|_{H^{s-1}})^{1+r-s}(\|\omega\|_{H^{s-1}}+\|\rho\|_{H^{s-2}})^{s-r}.
 \end{align}
Applying the energy estimate (\ref{e26}) gives
\begin{align}\label{e330}
 \|\omega\|_{H^{s}}+\|\rho\|_{H^{s-1}}
 &\leq \|u\|_{H^{s}}+\|\eta\|_{H^{s-1}}+\|v\|_{H^{s}}+\|\theta\|_{H^{s-1}}  \nonumber\\
 &\leq 2e^{C\widetilde{T}} (\|u_{0}\|_{H^{s}}+\|\eta_{0}\|_{H^{s-1}}+\|v_{0}\|_{H^{s}}+\|\theta_{0}\|_{H^{s-1}})  \nonumber\\
 &\leq 4e^{C\widetilde{T}}m.
 \end{align}
In addition, we claim that there exists a constant $K>0$ such that
\begin{equation}\label{e331}
\|\omega\|_{H^{s-1}}+\|\rho\|_{H^{s-2}}\leq e^{K}(\|\omega_{0}\|_{H^{s-1}}+\|\rho_{0}\|_{H^{s-2}}).
\end{equation}

Now let's focus on establishing (\ref{e331}). Applying $\Lambda^{s-1}$ to the first equation in (\ref{e328}), multiplying both sides by $\Lambda^{s-1}\omega$ and integrating over $R$ with respect to $x$, we obtain
\begin{align}\label{e332}
 \frac{d}{dt}\|\omega\|^{2}_{H^{s-1}(R)}
 &= -\langle \Lambda^{s-1}\partial_{x}(\omega(u+v)),\Lambda^{s-1}\omega \rangle  \nonumber\\
 &+2\langle  \Lambda^{s-3}\partial_{x}(\rho-\omega),  \Lambda^{s-1}\omega\rangle:=I_{1}+I_{2}.
 \end{align}
By integration by parts, H\"{o}lder inequality and Lemma \ref{lem2.3}, we have
\begin{align}\label{e333}
 |I_{1}|
 &=|\langle \Lambda^{s-1}\partial_{x}(\omega(u+v)),\Lambda^{s-1}\omega \rangle |  \nonumber\\
 &=|\langle [\Lambda^{s-1}\partial_{x},(u+v)]\omega,\Lambda^{s-1}\omega \rangle+\langle (u+v)\Lambda^{s-1}\partial_{x}\omega, \Lambda^{s-1}\omega \rangle | \nonumber\\
 &=|\langle [\Lambda^{s-1}\partial_{x},(u+v)]\omega,\Lambda^{s-1}\omega \rangle-\frac{1}{2}\langle \partial_{x}(u+v)\Lambda^{s-1}\omega, \Lambda^{s-1}\omega \rangle | \nonumber\\
 &\lesssim \|u+v\|_{H^{s}}\|\omega\|_{H^{s-1}}^{2}+\|(u+v)_{x}\|_{L^{\infty}}\|\omega\|_{H^{s-1}}^{2}\nonumber\\
 &\lesssim \|u+v\|_{H^{s}}\|\omega\|_{H^{s-1}}^{2},
 \end{align}
 and
 \begin{align}\label{e334}
 |I_{2}|
 &=2|\langle  \Lambda^{s-3}\partial_{x}(\rho-\omega),  \Lambda^{s-1}\omega\rangle |  \nonumber\\
 &\lesssim |\langle  \Lambda^{s-3}\partial_{x}\rho,\Lambda^{s-1}\omega \rangle|+|\langle  \Lambda^{s-3}\partial_{x}\omega,\Lambda^{s-1}\omega \rangle| \nonumber\\
 &\leq \|\rho\|_{H^{s-2}}\|\omega\|_{H^{s-1}}+\|\omega\|_{H^{s-2}}\|\omega\|_{H^{s-1}} \nonumber\\
 &\leq\|\rho\|_{H^{s-2}}\|\omega\|_{H^{s-1}}+\|\omega\|_{H^{s-1}}^{2}.
 \end{align}
Thus (\ref{e26}), (\ref{e332}), (\ref{e333}) and (\ref{e334}) imply
\begin{equation}\label{e335}
\frac{d}{dt}\|\omega\|_{H^{s-1}}\lesssim \|u+v\|_{H^{s}}\|\omega\|_{H^{s-1}}+\|\rho\|_{H^{s-2}}+\|\omega\|_{H^{s-1}}\leq C_{1}(m,\widetilde{T})\|\omega\|_{H^{s-1}}+\|\rho\|_{H^{s-2}}.
\end{equation}
On the other hand, applying $\Lambda^{s-2}$ to the second equation in (\ref{e328}), multiplying both sides by $\Lambda^{s-2}\rho$ and integrating over $R$ with respect to $x$, we obtain
\begin{align}\label{e336}
 \frac{d}{dt}\|\rho\|^{2}_{H^{s-2}(R)}
 &=-2\langle \Lambda^{s-2}\partial_{x}(u\rho+\theta\omega),\Lambda^{s-2}\rho \rangle \nonumber\\
 &= -2\langle \Lambda^{s-2}\partial_{x}(u\rho),\Lambda^{s-2}\rho \rangle-2\langle \Lambda^{s-2}\partial_{x}(\theta\omega),\Lambda^{s-2}\rho \rangle  \nonumber\\
 &:=I_{3}+I_{4}.
 \end{align}
By integration by parts, H\"{o}lder inequality and Lemma \ref{lem2.3}, we can get
\begin{align}\label{e337}
 |I_{3}|
 &=2|\langle \Lambda^{s-2}\partial_{x}(u\rho),\Lambda^{s-2}\rho \rangle |  \nonumber\\
 &=2|\langle [\Lambda^{s-2}\partial_{x},u]\rho,\Lambda^{s-2}\rho \rangle+\langle u\Lambda^{s-2}\partial_{x}\rho, \Lambda^{s-2}\rho\rangle | \nonumber\\
 &=2|\langle [\Lambda^{s-2}\partial_{x},u]\rho,\Lambda^{s-2}\rho \rangle-\frac{1}{2}\langle \partial_{x}u\Lambda^{s-2}\rho, \Lambda^{s-2}\rho \rangle | \nonumber\\
 &\lesssim \|u\|_{H^{s}}\|\rho\|_{H^{s-2}}^{2}+\|u_{x}\|_{L^{\infty}}\|\rho\|_{H^{s-2}}^{2}\nonumber\\
 &\lesssim \|u\|_{H^{s}}\|\rho\|_{H^{s-2}}^{2},
 \end{align}
 \begin{align}\label{e338}
 |I_{4}|
 &=2|\langle  \Lambda^{s-2}\partial_{x}(\theta\omega),  \Lambda^{s-2}\rho\rangle |  \nonumber\\
 &\lesssim |\langle  \Lambda^{s-2}\partial_{x}\theta\omega,\Lambda^{s-2}\rho \rangle|+|\langle  \Lambda^{s-2}\partial_{x}\omega\theta,\Lambda^{s-2}\rho \rangle| \nonumber\\
 &\leq (\|\Lambda^{s-2}\partial_{x}\theta\|_{L^{2}}\|\omega\|_{L^{\infty}}\|\Lambda^{s-2}\rho\|_{L^{2}}+
 \|\Lambda^{s-2}\partial_{x}\omega\|_{L^{2}}\|\theta\|_{L^{\infty}}\|\Lambda^{s-2}\rho\|_{L^{2}}) \nonumber\\
 &\lesssim \|\theta\|_{H^{s-1}}\|\omega\|_{H^{s-1}}\|\rho\|_{H^{s-2}}.
 \end{align}
Thus (\ref{e26}) and (\ref{e336})-(\ref{e339}) yield
\begin{equation}\label{e339}
\frac{d}{dt}\|\rho\|_{H^{s-2}}\lesssim \|u\|_{H^{s}}\|\rho\|_{H^{s-2}}+(\|\theta\|_{H^{s-1}}+1)\|\omega\|_{H^{s-1}}\leq C_{2}(m,\widetilde{T})(\|\omega\|_{H^{s-1}}+\|\rho\|_{H^{s-2}}).
\end{equation}
From (\ref{e335}) and (\ref{e339}), we can conclude that there exists a $K(m,\widetilde{T})$ such that
\begin{equation}\label{e340}
\frac{d}{dt}(\|\omega\|_{H^{s-1}}+\|\rho\|_{H^{s-2}})\lesssim K(m,\widetilde{T})(\|\omega\|_{H^{s-1}}+\|\rho\|_{H^{s-2}}),
\end{equation}
hence (\ref{e331}) follows from (\ref{e340}).

Due to $s-1\leq r<s$, we can apply (\ref{e329}), (\ref{e330}), (\ref{e331}) and Sobolev embedding to obtain
\begin{equation}
\begin{aligned}
 \|\omega\|_{H^{r}(R)}+\|\rho\|_{H^{r-1}(R)}
 &\leq 4(\|\omega\|_{H^{s}}+\|\rho\|_{H^{s-1}})^{1+r-s}(\|\omega\|_{H^{s-1}}+\|\rho\|_{H^{s-2}})^{s-r} \nonumber\\
 &\leq 4(4e^{C\widetilde{T}}m)^{1+r-s}e^{K(m,\widetilde{T})\beta}(\|\omega_{0}\|_{H^{s-1}}+\|\rho_{0}\|_{H^{s-2}})^{\beta} \nonumber\\
 &\leq 4(4e^{C\widetilde{T}}m)^{1+r-s}e^{K(m,\widetilde{T})\beta}(\|\omega_{0}\|_{H^{r}}+\|\rho_{0}\|_{H^{r-1}})^{\beta},
 \end{aligned}
\end{equation}
which finishes the proof of Theorem \ref{thm3.4}.
\end{proof}

\section{Ill-posedness in the critical space $H^{\frac{3}{2}}\times H^{\frac{1}{2}}$}
In this section, we mainly consider ill-posedness problem of the FW system (\ref{e12}) in critical space $H^{\frac{3}{2}}(R)\times H^{\frac{1}{2}}(R)$ and the ill-posedness is due to the norm inflation. Namely, there exist a solution to (\ref{e12}) which are initially arbitrarily small and eventually arbitrarily large with respect to the $H^{\frac{3}{2}}\times H^{\frac{1}{2}}$ norm, in an arbitrarily short time. More precisely, we have the following result.
\begin{theorem}\label{thm4.1}
For $\forall \varepsilon > 0$, there exists $(u_{0}, \eta_{0})\in H^{s}(R)\times H^{s-1}(R)$ with $s>\frac{3}{2}$ such that the following statements hold

 (1) $\|u_{0}\|_{H^{\frac{3}{2}}}\leq \varepsilon$ and $\|\eta_{0}\|_{H^{\frac{1}{2}}}\leq \varepsilon$;

 (2) There is a unique solution $(u, \eta)\in C([0, T); H^{s}\times H^{s-1})$ to the Cauchy problem of (\ref{e12}) with
 a maximal lifespan $T < \varepsilon$;

 (3) Either
$$
\limsup_{t\rightarrow T^{-}}\|u(t)\|_{H^{\frac{3}{2}}}\geq \limsup_{t\rightarrow T^{-}}\|u(t)\|_{B_{2, \infty}^{\frac{3}{2}}}= \infty
$$
or
$$
\limsup_{t\rightarrow T^{-}}\|\eta(t)\|_{H^{\frac{1}{2}}}\geq \limsup_{t\rightarrow T^{-}}\|\eta(t)\|_{B_{\infty, \infty}^{0}}= \infty
$$
occurs.
\end{theorem}

Before proving the Theorem \ref{thm4.1}, let's first introduce two useful lemmas. In the following, we mainly consider the FW system along the flow $q(t,x)$ generated by $u$, that is to say,
 \begin{equation}\label{e41}
 \left\{ \begin{array}{ll}
 \frac{\partial q(t,x)}{\partial t}=u(t,q(t,x)),~~(t,x)\in [0,T)\times R, \\
 q(0,x)=x,
 \end{array} \right.
 \end{equation}
where there exists a unique solution $q\in C([0,T)\times R)$ to (\ref{e41}) such that
\begin{equation}\label{e42}
q_{x}(t,x)=e^{\int^{t}_{0}u_{x}(s,q(s,x))ds}>0, ~~for~~(t,x)\in [0,T)\times R.
\end{equation}
From \cite{9}, a simple computation implies
\begin{equation}\label{e43}
\eta(t,q(t,x))q_{x}(t,x)=\eta_{0}(x),
\end{equation}
where $(u,\eta)$ is the solution to FW system with the initial data $(u_{0},\eta_{0})$.
\begin{lemma}\label{lem4.2}
Let $(u_{0},\eta_{0})\in H^{s}(R)\times H^{s-1}(R)$ with $s>\frac{3}{2}$ and $T$ is the maximal existence time of the corresponding solution $(u,\eta)$ to FW system (\ref{e12}). For any $t\in [0,T)$, we have the following conservations
\begin{equation}\label{e44}
\int_{R}udx=\int_{R}u_{0}dx, \quad \int_{R}\eta dx=\int_{R}\eta_{0}dx.
\end{equation}
Moreover, if $\eta_{0}\geq 0$, we have
\begin{equation}\label{e45}
\|u\|_{L^{2}}\leq \|u_{0}\|_{L^{2}}+\frac{1}{2}\|\eta_{0}\|_{L^{1}}t
\end{equation}
and
\begin{equation}\label{e46}
\|u\|_{L^{\infty}}\leq \|u_{0}\|_{L^{\infty}}+(\|u_{0}\|_{L^{2}}+\frac{1}{2}\|\eta_{0}\|_{L^{1}})t+\frac{1}{2}\|\eta_{0}\|_{L^{1}}t^{2}.
\end{equation}
\end{lemma}
\begin{proof}
By the system (\ref{e12}) and integration by parts, we have
$$
\frac{d}{dt}\int_{R}udx=-\frac{1}{2}\int_{R}(u^{2})_{x}dx+\int_{R}\Lambda^{-2}\eta_{x}dx-\int_{R}\Lambda^{-2}u_{x}dx=0
$$
and
$$
\frac{d}{dt}\int_{R}\eta dx=-\int_{R}(u\eta)_{x}dx=0,
$$
which give (\ref{e44}). From (\ref{e43}), (\ref{e44}) and Young's inequality, we have
\begin{align}\label{e47}
 \frac{1}{2}\frac{d}{dt}\|u\|_{L^{2}}^{2}
 &=\frac{1}{2}\int_{R}u^{2}u_{x}dx+\int_{R}\Lambda^{-2}\eta_{x}udx-\int_{R}\Lambda^{-2}u_{x}udx \nonumber\\
 &=\int_{R}\Lambda^{-2}\eta_{x}udx\leq \|\partial_{x}\Lambda^{-2}\eta\|_{L^{2}}\|u\|_{L^{2}}=\frac{1}{2}\|\partial_{x}e^{-|x|}\ast\eta\|_{L^{2}}\|u\|_{L^{2}} \nonumber\\
 &\leq \frac{1}{2}\|sgn(x)e^{-|x|}\|_{L^{2}}\|\eta\|_{L^{1}}\|u\|_{L^{2}}\leq \frac{1}{2}\|\eta_{0}\|_{L^{1}}\|u\|_{L^{2}}.
 \end{align}
Hence (\ref{e45}) follows from (\ref{e47}) by integrating from $0$ to $t$ with respect to time variable. In addition, along the flow $q(t,x)$, we can obtain
\begin{align}\label{e48}
 \frac{du}{dt}
 &=(\partial_{t}u+uu_{x})(t,q(t,x))=\partial_{x}\Lambda^{-2}(\eta-u)(t,q(t,x)) \nonumber\\
 &\leq \frac{1}{2}\|sgn(x)e^{-|x|}\|_{L^{\infty}}\|\eta\|_{L^{1}}+\frac{1}{2}\|sgn(x)e^{-|x|}\|_{L^{2}}\|u\|_{L^{2}} \nonumber\\
 &\leq \frac{1}{2}\|\eta_{0}\|_{L^{1}}+\|u_{0}\|_{L^{2}}+\frac{1}{2}\|\eta_{0}\|_{L^{1}}t,
 \end{align}
where we use (\ref{e44}), (\ref{e45}) and Young's inequality. Integrating from $0$ to $t$ on both sides of (\ref{e48}) yields
$$
\|u\|_{L^{\infty}}\leq \|u_{0}\|_{L^{\infty}}+(\|u_{0}\|_{L^{2}}+\frac{1}{2}\|\eta_{0}\|_{L^{1}})t+\frac{1}{2}\|\eta_{0}\|_{L^{1}}t^{2}.
$$
\end{proof}
 \begin{lemma}\label{lem4.3}
 Let $(u_{0},\eta_{0})\in H^{s}(R)\times H^{s-1}(R)$ with $s>\frac{3}{2}$ and $\eta_{0}\geq 0$, assume there is $x_{0}\in R$ such that
 \begin{equation}\label{e49}
 u_{0}'(x_{0})< -2~~and~~ u_{0}'(x_{0})^{2}> 4(\|u_{0}\|_{L^{2}}+\|u\|_{L^{\infty}}+\|\eta_{0}\|_{L^{1}}),
 \end{equation}
 then the solution of (\ref{e12}) blows up in finite time. Moreover, the lifespan can be estimated by
 \begin{equation}\label{e410}
  T \leq -\frac{2}{u_{0}'(x_{0})}<1.
 \end{equation}
 \end{lemma}
\begin{proof}
Assume that $T > 0$ be maximal existence time of the solution $(u, \eta)$ to (\ref{e12}) and let
$$
m(t)= u_{x}(t, q(t, x_{0}))
$$
along the flow $q(t, x)$. Combining the first equation in FW system (\ref{e12}) with Lemma \ref{lem4.2}, we have
\begin{align}\label{e411}
&\partial_{t}m(t)+ m^{2}(t)= (\partial_{t}u_{x}+uu_{xx}+u_{x}^{2})(t,q(t,x_{0})) \nonumber \\
&=\partial_{x}^{2}\Lambda^{-2}(\eta-u)(t,q(t,x_{0}))=(I-\partial_{x}^{2})^{-1}(\eta-u)(t,q(t,x_{0}))-(\eta-u)(t,q(t,x_{0}))  \nonumber \\
&\leq \frac{1}{2}\|e^{-|x|}\ast\eta\|_{L^{\infty}}+\frac{1}{2}\|e^{-|x|}\ast u\|_{L^{\infty}}+\|u\|_{L^{\infty}}\leq \frac{1}{2}\|\eta\|_{L^{1}}+\|u\|_{L^{2}}+\|u\|_{L^{\infty}}  \nonumber \\
&\leq \frac{1}{2}\|\eta_{0}\|_{L^{1}}+\|u_{0}\|_{L^{2}}+\|u_{0}\|_{L^{\infty}}
+(\|u_{0}\|_{L^{2}}+\|\eta_{0}\|_{L^{1}})t+\frac{1}{2}\|\eta_{0}\|_{L^{1}}t^{2},
\end{align}
where we use the fact $\eta(t)$ keeps the sign along the flow $q(t,x)$. Define
$$
 M(t)= \frac{1}{2}\|\eta_{0}\|_{L^{1}}+\|u_{0}\|_{L^{2}}+\|u_{0}\|_{L^{\infty}}
  +(\|u_{0}\|_{L^{2}}+\|\eta_{0}\|_{L^{1}})t+\frac{1}{2}\|\eta_{0}\|_{L^{1}}t^{2}.
$$

By $(\ref{e49})$, we can choose
$$
T_{1}= \frac{-(\|u_{0}\|_{L^{2}}+\|\eta_{0}\|_{L^{1}})+
\sqrt{\|\eta_{0}\|_{L^{1}}m^{2}(0)+\|u_{0}\|_{L^{2}}^{2}-2\|\eta_{0}\|_{L^{1}}\|u_{0}\|_{L^{\infty}}}}{\|\eta_{0}\|_{L^{1}}}
$$
bigger than 1, such that
\begin{equation}\label{e412}
M(T_{1})\leq \frac{1}{2} m^{2}(0)< m^{2}(0).
\end{equation}
Standard arguments on continuity yields
\begin{equation}\label{e413}
M(t)<  m^{2}(t),~~for~~t\in[0,T_{1}]\cap[0,T).
\end{equation}
In view of (\ref{e411}) and (\ref{e413}), we obtain for $t\in[0,T_{1}]\cap[0,T)$
\begin{equation}\label{e414}
\partial_{t}m(t)\leq -m(t)^{2}+M(t)<-m(t)^{2}+m(t)^{2}=0,
\end{equation}
which means that $m(t)$ is a decreasing function over $t\in[0,T_{1}]\cap[0,T)$. Hence we have
\begin{equation}\label{e415}
m(t)<m(0)<0,~~for~~t\in[0,T_{1}]\cap[0,T).
\end{equation}

In addition, from (\ref{e412}) and (\ref{e415}), we can deduce
\begin{equation}\label{e416}
M(T_{1})< \frac{1}{2} m^{2}(0)<\frac{1}{2} m^{2}(t),~~for~~t\in[0,T_{1}]\cap[0,T).
\end{equation}
Hence (\ref{e411}) and (\ref{e416}) imply that there holds for $t\in[0,T_{1}]\cap[0,T)$
\begin{equation}\label{e417}
\partial_{t}m(t)\leq -m(t)^{2}+M(t)<-m(t)^{2}+M(T_{1})<-m(t)^{2}+\frac{1}{2}m(t)^{2}=-\frac{1}{2}m(t)^{2}.
\end{equation}
From (\ref{e49}), (\ref{e415}) and (\ref{e417}), we attain
$$
m(t)\rightarrow-\infty,~~as~~t\rightarrow-\frac{2}{m(0)}.
$$
\end{proof}

Now we are in the position to prove the Theorem \ref{thm4.1}, where the proof is based on a contradiction argument by constructing the special initial data.
\begin{proof}
Let $\mathcal{C}$ be an interval included in $(\frac{1}{2},1)$ with $\mathcal{C}\cap 2\mathcal{C}=\emptyset$. Define
$$
P(x)=\sum_{j\geq 1}\mathcal{F}^{-1}(\frac{i\xi}{j\cdot2^{3j}}\mathbf{1}_{2^{j}\mathcal{C}})(x)
$$
and
$$
P_{\leq N}(x)=\sum_{j=1}^{N}\mathcal{F}^{-1}(\frac{i\xi}{j\cdot2^{3j}}\mathbf{1}_{2^{j}\mathcal{C}})(x).
$$
It's easy to see that
\begin{align}\label{e418}
\|P(x)\|_{H^{\frac{3}{2}}}^{2}
&\leq \sum_{j\geq 1}\int_{R}(1+|\xi|^{2})^{\frac{3}{2}}\frac{|\xi|^{2}}{j^{2}\cdot2^{6j}}\mathbf{1}_{2^{j}\mathcal{C}} d\xi=\sum_{j\geq 1}\int_{2^{j}\mathcal{C}}(1+|\xi|^{2})^{\frac{3}{2}}\frac{|\xi|^{2}}{j^{2}\cdot2^{6j}} d\xi \nonumber \\
&=\sum_{j\geq 1}\frac{1}{j^{2}\cdot2^{6j}}\int_{2^{j}\mathcal{C}}(1+|\xi|^{2})^{\frac{3}{2}}\xi^{2}d\xi \lesssim
\sum_{j\geq 1}\frac{1}{j^{2}\cdot2^{6j}} 2^{6j} \nonumber \\
&=\sum_{j\geq 1}\frac{1}{j^{2}}=\frac{\pi^{2}}{6}.
\end{align}
Similarly, for a fixed $N\geq1$ and any $s\geq 0$, we have
\begin{equation}\label{e419}
\|P_{N}(x)\|_{H^{s}}^{2} \lesssim
\sum_{j= 1}^{N}\frac{1}{j^{2}\cdot2^{6j}} 2^{(3+2s)j}
=\sum_{j\geq 1}^{N}\frac{2^{(2s-3)j}}{j^{2}}<\infty.
\end{equation}
Therefore, for $\varepsilon>0$, let's define
\begin{equation}\label{e420}
u_{0,\varepsilon}=\frac{P_{\leq N}(x)\varepsilon}{\|P\|_{H^{\frac{3}{2}}}}, \quad  \eta_{0,\varepsilon}=\frac{\varphi(x)\varepsilon}{\|\varphi\|_{H^{\frac{1}{2}}}},
\end{equation}
where $\varphi\in \mathcal{S}(R)$ and $\varphi(x)\geq 0$.
Then it's obvious that
$$
\|u_{0,\varepsilon}\|_{H^{\frac{3}{2}}}\leq \varepsilon, \quad \|\eta_{0,\varepsilon}\|_{H^{\frac{1}{2}}}\leq \varepsilon,
$$
$(u_{0,\varepsilon},\eta_{0,\varepsilon})\in H^{s}(R)\times \mathcal{S}(R)$ and $\eta_{0,\varepsilon}\geq 0$.
On the other hand, we know
\begin{align}\label{e421}
P'(0)
&=\int_{R}\mathcal{F}(\partial_{x}P)(\xi)d\xi \nonumber \\
&=\sum_{j\geq 1}\frac{1}{j\cdot2^{3j}}\int_{R}(i\xi)^{2}\mathbf{1}_{2^{j}\mathcal{C}} d\xi  \nonumber \\
&=-\sum_{j\geq 1}\frac{1}{j\cdot2^{3j}}\int_{2^{j}\mathcal{C}}\xi^{2}d\xi \nonumber \\
&\lesssim -\sum_{j\geq 1}\frac{1}{j\cdot2^{3j}}2^{3j}= -\sum_{j\geq 1}\frac{1}{j}=-\infty.
\end{align}
From (\ref{e421}), we have
$$
u_{0,\varepsilon}'<-\frac{2}{\varepsilon}
$$
by choosing $N$ sufficiently large. By the Lemma \ref{lem4.3}, for the given initial data $(u_{0,\varepsilon},\eta_{0,\varepsilon})$, there is a unique solution $(u_{\varepsilon},\eta_{\varepsilon})\in C([0,T);H^{s}\times H^{s-1})$ for $s>\frac{3}{2}$ with the lifespan $T_{\varepsilon}\leq-\frac{2}{u_{0,\varepsilon}'}<\varepsilon$. Now it remains to show that either
\begin{equation}\label{e422}
\limsup_{t\rightarrow T_{\varepsilon}}\|u_{\varepsilon}\|_{B^{\frac{3}{2}}_{\infty, \infty}}=\infty
\end{equation}
or
\begin{equation}\label{e423}
\limsup_{t\rightarrow T_{\varepsilon}}\|\eta_{\varepsilon}\|_{B^{0}_{\infty, \infty}}=\infty
\end{equation}
happens. We prove the fact by a contradiction argument. Suppose that neither (\ref{e422}) nor (\ref{e423}) occurs, then there exists a constant $M_{\varepsilon}>0$ such that
\begin{equation}\label{e424}
\limsup_{t\rightarrow T_{\varepsilon}}\|u_{\varepsilon}(t)\|_{B^{\frac{3}{2}}_{\infty, \infty}}\leq M_{\varepsilon},\quad \limsup_{t\rightarrow T_{\varepsilon}}\|\eta_{\varepsilon}(t)\|_{B^{0}_{\infty, \infty}}\leq M_{\varepsilon}.
\end{equation}
The energy estimate in Lemma \ref{lem2.10} and the inequality in Lemma \ref{lem2.8} yield
\begin{align}\label{e425}
&\frac{d}{dt}(\|u_{\varepsilon}\|_{H^{2}}+\|\eta_{\varepsilon}\|_{H^{1}}) \nonumber \\
&\leq C(\|\partial_{x}u_{\varepsilon}\|_{B^{\frac{1}{2}}_{2,\infty}\cap L^{\infty}}+\|\eta_{\varepsilon}\|_{L^{\infty}}+1)(\|u_{\varepsilon}\|_{H^{2}}+\|\eta_{\varepsilon}\|_{H^{1}}) \nonumber \\
&\leq C(1+ \|u_{\varepsilon}(t)\|_{B_{2,\infty}^{\frac{3}{2}}}log(e+\|u_{\varepsilon}(t)\|_{H^{2}}) \nonumber \\
&+\|\eta_{\varepsilon}(t)\|_{B_{\infty,\infty}^{0}}log(e+\|\eta_{\varepsilon}(t)\|_{H^{1}}))
 (\|u_{\varepsilon}\|_{H^{2}}+\|\eta_{\varepsilon}\|_{H^{1}})  \nonumber \\
&\leq CM_{\varepsilon}(\|u_{\varepsilon}\|_{H^{2}}+\|\eta_{\varepsilon}\|_{H^{1}})
 log(e+\|u_{\varepsilon}(t)\|_{H^{2}}+\|\eta_{\varepsilon}(t)\|_{H^{1}}).
\end{align}
Based on (\ref{e425}), we can use the Gronwall inequality in Lemma \ref{lem2.5} to obtain
$$
\sup_{t\in[0,T_{\varepsilon}]}\|u_{\varepsilon}\|_{H^{2}}<\infty,
$$
which is in contradiction with the blow-up result in Lemma \ref{lem4.3}. Thus, either (\ref{e422}) or (\ref{e423}) would happen and the proof is completed.
\end{proof}

    \bigskip

\noindent \textbf{Acknowledgement.} {This project is supported by National Natural Science Foundation of China (No:11571057).}

\end{document}